\documentclass[11pt]{amsart}
\usepackage{amsmath,amsfonts,amssymb,amsthm,amscd,latexsym,euscript,verbatim}
\usepackage[all]{xy}

\makeatletter
\def\section{\@startsection{section}{1}%
  \z@{1.1\linespacing\@plus\linespacing}{.8\linespacing}%
  {\normalfont\Large\scshape\centering}}
\makeatother

\theoremstyle{plain}

\newtheorem*{MT}{Main Theorem}

\newtheorem*{conj*}{Root Groups Conjecture}
\newtheorem*{thm1.2}{(1.2) Theorem}
\newtheorem*{thm1.3}{(1.3) Theorem}
\newtheorem*{thm1.4}{(1.4) Theorem}
\newtheorem*{prop*}{Proposition}

\newtheorem{prop}{Proposition}[section]

\newtheorem{lemma}[prop]{Lemma}

\theoremstyle{definition}

\newtheorem*{Def*}{Definition}

\newtheorem*{notation*}{Notation}
\newtheorem{remark}[prop]{Remark}

\newtheorem*{remark*}{Remark}

\newcommand{\ff}{\mathbb{F}}

\newcommand{\ga}{\alpha}

\newcommand{\charc}{{\rm char}}

\newcommand{\sminus}{\smallsetminus}

\newcommand{\one}{{\bf 1}}

\numberwithin{equation}{section}

\begin{document}
\title[]{A characterization of the quaternions using commutators}
\author[Kleinfeld and Segev]{Erwin Kleinfeld\qquad Yoav Segev}

\address{Erwin Kleinfeld \\
1555 N.~Sierra St.~Apt 120, Reno, NV 89503-1719, USA}
\email{erwinkleinfeld@gmail.com}

\address{Yoav Segev \\
         Department of Mathematics \\
        Ben-Gurion University \\
        Beer-Sheva 84105 \\
         Israel}
\email{yoavs@math.bgu.ac.il}


\begin{abstract}
Let $R$ be an associative  ring with $\one,$ which is not commutative.
Assume that any non-zero commutator $v\in R$ satisfies: $v^2$
is in the center of $R$ and $v$ is not a zero-divisor. 
(Note that our assumptions do not include finite dimensionality.)

We prove that $R$ has no zero divisors,
and that if $\charc(R)\ne 2,$ then the localization of $R$
at its center is a quaternion division algebra.  
\end{abstract}

\date{\today}
\maketitle

%%%%%%%%%%%%%%%%%%%%%%%%%%%%%%%%%%%%%%%%%%
%%%%%%%%%%%%%%%%%%%%%%%%%%%%%%%%%%%%%%%
%%%%%%%%%%%%%%%%%%%%%%%%%%%%%%%%%%%%%%
%
\section{Introduction}
%%%%%%%%%%%%%%%%%%%%%%%%%%%%%%%%%%%%%%
%%%%%%%%%%%%%%%%%%%%%%%%%%%%%%%%%%%%%%%
%%%%%%%%%%%%%%%%%%%%%%%%%%%%%%%%%%%%%%
In a recent paper \cite{KS} we charaterized the octonions
amongst alternative rings.  In this paper we work rather
similarly with associative rings and surprisingly get
a characterization of quaternion algebras in much
the same way, and in the spirit of \cite{BK}.

Let $D$ be a quaterion division algebra over a field $\ff.$
Thus $D=\ff +\ff i +\ff j+ \ff k,$  with $i^2, j^2\in \ff,$  
and $k=ij=-ji.$ A pure quaternion is an element  $p\in D$ such that
$p\in \ff i+\ff j+\ff k.$  It is easy to check that $p^2\in\ff,$ for
a pure quaternion $p,$ and that given
$x,y\in D,$ the commutator $(x,y)=xy-yx$ is a pure quaternion.

In this note we show that this characterizes the quaternion division algebras.
(Note, we do not assume finite dimensionality.)

%%%%%%%%%%%%%%%%%%%%%%%%%%%%%%%%%%%%%%%%%%%%%
\begin{MT}
%%%%%%%%%%%%%%%%%%%%%%%%%%%%%%%%%%%%%%%%%%%%
Let $R$ be an associative ring with $\one$ which is not commutative such that
\begin{itemize}
\item[(i)]
A non-zero  commutator in $R$  is not a divisor of zero in $R;$

\item[(ii)]
$(x,y)^2\in C,$ for all $x,y\in R,$ where $C$ is the center of $R.$
\end{itemize}
Then
\begin{enumerate}
\item
$R$ contains no divisors of zero.

\item
If, in addition, the characteristic of $R$ is not $2,$
then the localization of $R$ at $C$ is a quaternion division algebra,
whose center is the fraction field  of $C.$
\end{enumerate}
\end{MT}

We note that if $x,y\in R$ are non-zero elements such that $xy=0,$
then we say that both $x$ and $y$ are zero divisors in $R.$

%%%%%%%%%%%%%%%%%%%%%%%%%%%%%%%%%%%%%%%%%%%%%%%%%%%%%%%%%%%
%%%%%%%%%%%%%%%%%%%%%%%%%%%%%%%%%%%%%%%%%%%%%%%%%%%%%%%
%%%%%%%%%%%%%%%%%%%%%%%%%%%%%%%%%%%%%%%%%%%%%%%%%%%%%%%%
%
\section{Proof of the Main Theorem}
%%%%%%%%%%%%%%%%%%%%%%%%%%%%%%%%%%%%%%%%%%%%%%%%%%%%%
%%%%%%%%%%%%%%%%%%%%%%%%%%%%%%%%%%%%%%%%%%%%%%%%%%
%%%%%%%%%%%%%%%%%%%%%%%%%%%%%%%%%%%%%%%%%%%%%%%

In this section $R$ is an associative ring with $\one$ which is not commutative.
We denote by $C$ the center of $R.$
We assume that: 
\[
\begin{aligned}
&(1)\quad \text{Non-zero commutators are not divisors of zero in $R.$}\\
&(2)\quad \text{Squares of commutators in $R$ are in $C.$}
\end{aligned}
\]

%%%%%%%%%%%%%%%%%%%%%%%%%%%%%%%%%%%%%%%%%%%%%%%
%
\begin{lemma}\label{lem C}
%%%%%%%%%%%%%%%%%%%%%%%%%%%%%%%%%%%%%%%%%%%%%%
If $c\in C,$ then $c$ is not a zero divisor in $R.$
\end{lemma}
\begin{proof}
Suppose $cr=0,$ and let $v:=(x,y)$ be a non-zero commutator.
Then $(vc)r=0,$ but $vc=(x,yc)\ne 0,$ a contradiction.
\end{proof}

%%%%%%%%%%%%%%%%%%%%%%%%%%%%%%%%%%%%%%%%
%
\begin{prop}\label{prop quadratic}
%%%%%%%%%%%%%%%%%%%%%%%%%%%%%%%%%%%%%%%
Let $x\in R\sminus C,$ and let $v=(x,y)$ be a non-zero commutator.
Then
\begin{enumerate}
\item
 $v+vx$ and $vx$ are commutators.

\item
$ax^2+bx+c=0,$ for some $a,b,c\in C,$ with $a, c$ non-zero.

\item
$x$ is not a divisor of zero in $R.$

\item
$R$ contains no zero divisors.
\end{enumerate}
\end{prop}
\begin{proof}
(1)\quad
We have $v+vx=v(\one+x)=(x,y(\one+x)),$  and $vx=(x,yx).$

(2)\quad
Let $\ga:=(v+vx)^2=v^2+v^2x+vxv+(vx)^2.$ Then $\ga\in C.$
We have $\ga x=(v^2+(vx)^2)x+v^2x^2+(vx)^2.$  Letting $a:=v^2, b:=v^2+(vx)^2-(v+vx)^2$ and $c:=(vx)^2,$
we see that $a, b, c\in C,$ and $ax^2+bx+c=0,$ with $a\ne 0\ne c.$

(3)\quad
Suppose that $xy=0,$ for some non-zero $y\in R,$ then we immediately
get that $cy=0,$ contradicting Lemma \ref{lem C}.

(4)\quad This follows from (3) and Lemma \ref{lem C}.
\end{proof}

%%%%%%%%%%%%%%%%%%%%%%%%%%%%%%%%%%%%%%%%%%
%
\begin{remark}
%%%%%%%%%%%%%%%%%%%%%%%%%%%%%%%%%%%%%%%%
In view of Proposition \ref{prop quadratic}(4), we can
form the {\it localization of $R$ at $C,$}  $R//C.$ This is the
set  of  all  formal fractions $x/c,\  x\in R,\ c\in C,\ c\ne 0,$ 
with the  obvious definitions: (i)  $x/c=y/d$ if and  only if $dx = cy;$  
(ii)  $(x/c) +  (y/d)= (dx+cy)/(cd);$ (iii)  $(x/c)(y/d)=  (xy)/(cd).$ 
It  easy to check that $r\mapsto r/\one$ is an embedding
of $R$ into $R//C$ and that the center of
$R//C$ is the fraction field of $C.$  Thus {\bf from now on
we replace $R$ with $R//C$ and assume that $C$ is a field.} 
\end{remark}

%%%%%%%%%%%%%%%%%%%%%%%%%%%%%%%%%%%%%%%%%%%%%%%%
%
\begin{lemma}\label{lem Q}$ $
%%%%%%%%%%%%%%%%%%%%%%%%%%%%%%%%%%%%%%%%%%%%%
\begin{enumerate}
\item
There exists a comutator $i:=(x,y)$ which is not in $C.$

\item
For $i$ as in (1), let $j:=(i,s)$ be nonzero.  Then $ij=-ji.$

\item
Let $k:=ij.$  Then ${\bf Q}:=C+Ci+Cj+Ck$ is a quaternion division algebra.
\end{enumerate}
\end{lemma}
\begin{proof}
(1)\quad
Let $x\in R\sminus C,$ and let $v:=(x,y)\ne 0.$  Suppose that $v\in C,$
then $vx\notin C,$ and $vx=(x,yx).$

(2)\quad
Since $i\notin C,$ there is $s\in R,$ with $j:=(i,s)\ne 0.$
But then 
\[
ij=i(is-si)=i^2s-isi=-(isi-si^2)=-ji.
\]

(3)\quad
Since $i^2, j^2\in C,$ and $ij=-ji,$ part (3) holds.
\end{proof}
From now on we let
\[
{\bf Q}=C+Ci+Cj+Ck,\quad\text{as in Lemma \ref{lem Q}.}
\]

%%%%%%%%%%%%%%%%%%%%%%%%%%%%%%%%%%%%%
%
\begin{prop}\label{prop main}
%%%%%%%%%%%%%%%%%%%%%%%%%%%%%%%%%%%%%
Assume that $\charc(C)\ne 2.$ Then
\begin{enumerate}
\item
If $p\in R$ satisfies

$\qquad (*)\quad
pu+up=d_u\in C,\quad\text{for all }u\in\{i,j,k\},$

\noindent
then $p\in {\bf Q}.$ 
\item
If $R\ne {\bf Q},$ then
there exists $p\in R\sminus {\bf Q}$ satisfying  $(*)$ above.

\item
R={\bf Q}.
\end{enumerate}
\end{prop}
\begin{proof}
(1)\quad
We proceed as in \cite{K}, p.~140. Set
\[
m=p-(d_i/2i^2)i-(d_j/2j^2)j-(d_k/2k^2)k.
\]
Then
\[
mi+im=
pi+ip-d_i-(d_j/2j^2)ji-(d_j/2j^2)ij-(d_k/2k^2)ki-(d_k/2k^2)ik=0.
\]
Similarly $mj+jm=0=mk+km.$

But then
\[
0=mk+km=mij+ijm=2ijm=2km.
\]
Since $\charc(C)\ne 2,$ and $R$ has no zero divisors
we must have $m=0,$ so $p\in {\bf Q}.$

(2)\quad
Let $x\in R\sminus {\bf Q}.$  
By Proposition \ref{prop quadratic}(2), $x$ satisfies a quadratic, and hence a monic quadratic
equation $x^2-bx+c=0.$  As in \cite{K}, p.~140, let $p:=x-b/2.$ Then $p\notin {\bf Q},$
and $p^2\in C.$  Let $u\in\{i,j,k\}.$ Then both $p+u$ and $p-u$ satisfy
a quadratic equation over $C.$  That is
\[
\begin{aligned}
&(p+u)^2=c_1(p+u)+c_2\\
&(p-u)^2=c_3(p-u)+c_4.
\end{aligned}
\]
Adding we get
\[
(c_1+c_3)p+(c_1-c_3)u+c_5=0,\quad\text{where }c_5=c_2+c_4-2p^2-2u^2\in C.
\]
Now $c_1+c_3=0,$ since $p\notin {\bf Q},$  and then $c_1-c_3=0,$ since $u\notin C.$
We thus get that
\[
pu+up=c_2-p^2-u^2\in C.
\]

(3)\quad
This follows from (1) and (2).
\end{proof}

\noindent
\begin{proof}[\bf Proof of the Main Theorem.]  Part (1) follows from 
and Proposition \ref{prop quadratic}(4), and  part (2) follows from Proposition \ref{prop main}(3).
\end{proof}

%%%%%%%%%%%%%%%%%%%%%%%%%%%%%%%%%%%%%
%%%%%%%%%%%%%%%%%%%%%%%%%%%%%%%%%%%%%%%
%%%%%%%%%%%%%%%%%%%%%%%%%%%%%%%%%%%%%%%%

\end{document}